\documentclass[pdflatex,sn-mathphys-ay]{sn-jnl}
\usepackage{graphicx}
\usepackage{mathtools}
\usepackage{placeins}
\usepackage{url}
\usepackage{amsmath}
\usepackage{float}      
\usepackage{tabularx}   
\usepackage{booktabs}   
\usepackage{mathrsfs} 
\usepackage{amsfonts} 
\usepackage{dsfont}
\usepackage{lipsum}
\usepackage{longtable}
\usepackage{subcaption}
\usepackage{amsthm}%
\usepackage{longtable}
\usepackage{caption} 
\usepackage[utf8]{inputenc} 
\usepackage{csquotes}
\usepackage{multirow}
\usepackage{hyperref}
\usepackage{natbib}
\usepackage[T1]{fontenc}
\usepackage[utf8]{inputenc} 
\usepackage{csquotes}
\newtheorem{theorem}{Theorem}
\newtheorem{lemma}[theorem]{Lemma}

\theoremstyle{thmstyleone}%
\theoremstyle{thmstyletwo}%

\theoremstyle{thmstylethree}%
\newtheorem{definition}{Definition}%

\begin{document}
\newgeometry{margin=1in}
\title[Tail Dependence]{Measuring Tail Dependence in Linear Processes: Theory and Empirics}


\author*[1]{\fnm{Debanjana} \sur{Datta}
}

\author[2]{\fnm{Diganta} \sur{Mukherjee}}

\affil[1]{\orgname{Indian Statistical Institute}, \city{Bengaluru}}
\affil[2]{\orgname{Indian Statistical Institute}, \city{Kolkata}}
\renewcommand\thefootnote{}\footnotetext{*Corresponding author: debnjanad@gmail.com}


\abstract{The quantitative analysis of financial time series often reveals two distinct features that standard Gaussian frameworks fail to capture: heavy-tailed marginal distributions and the phenomenon of extreme co-movements.While extreme value theory characterizes marginal behavior, Copulas provide a functional bridge to describe the dependence structure independently of the marginals. We are proposing a different way of looking at the joint extremes on the basis of a dependence measure. The proposed idea incorporates both the non-identical and identical regularly varying distributions. Informed by the analysis of some high-frequency cryptocurrency datasets, the effect of persistence property have been thoroughly studied under these setups. A detailed simulation study confirms our intuition and findings.}




\maketitle

\section{Introduction}\label{sec1}
For a dynamic stochastic process, temporal dependence is usually studied using serial correlation. But an equally important dimension is clustering of exceedance at high thresholds.
This is a well-known phenomenon in econometric modeling where ARCH GARCH models are used. Correlation based dependence measures mainly refer to deviations around the mean whereas tail dependence or extremal dependence takes into account the joint extremes of the underlying distributions.

The modern financial landscape is increasingly characterized by \enquote{Black Swan} events—extreme market movements that occur with a frequency far exceeding the predictions of the Classical Central Limit Theorem. 
In finance, risk is traditionally quantified through measures such as \textbf{Value-at-Risk (VaR)} and \textbf{Expected Shortfall (ES)}. However, standard parametric estimations of these metrics often fail during periods of high contagion because they rely on linear correlation and thin-tailed distributions. EVT transcends these limitations by focusing exclusively on the tail regions. Some notable examples are:
\begin{itemize}
\item In corporate finance, a Collateralized Debt Obligation (CDO) is a complex structured financial product that pools together a cash-flow-generating set of assets—such as mortgages, bonds, or loans—and repackages them into discrete tranches based on risk and return. Extreme Value Theory(EVT) comes into rescue at the modeling of joint defaults—specifically, how the probability of multiple assets failing simultaneously increases during market stress. Insightful details regarding this can be found in \cite{garcia2003applications}, \cite{gulati2019analysis} and others.

\item In simple terms, reinsurance is \enquote{insurance for insurance companies}. It is a contract where one insurance company (the reinsurer) agrees to indemnify another insurance company (the cedant) for a portion of the losses it may incur under its active policies. This mechanism is a cornerstone of the global financial system because it prevents a single massive disaster. In the context of reinsurance, Extreme Value Theory (EVT) provides the statistical backbone for modeling \enquote{Ruin Probability} and pricing catastrophic layers. While standard insurance relies on the Law of Large Numbers, reinsurance deals with the Principle of a Single Big Jump, where one extreme event can dominate the total loss. For an overview, one can refer to \cite{morales2004approximation}, \cite{aviv2018extreme}, \cite{jung1964use} and others.

\item The catastrophic impact of Hurricane Katrina (2005) is a primary case study in Multivariate Extreme Value Theory (MEVT), specifically regarding how uncorrelated risks can suddenly become highly dependent during a \enquote{tail event}. \cite{liu2006joint}, \cite{gulati2019analysis} and others have provided insightful details in this area. In the context of Hurricane Katrina, the failure of the New Orleans levee system is a classic example of Tail Dependence. From a standard risk management perspective, a \enquote{Wind} event (Hurricane) and a \enquote{Flood} event (Levee Breach) were modeled as separate, relatively uncorrelated risks.
\end{itemize} 

To rigorously quantify these risks, we frequently appeal to concepts like \textbf{Regularly Varying Distributions} and \textbf{Extreme Value Theory (EVT)}. While traditional finance relies on the assumption of finite moments and normal innovations, empirical evidence from high-frequency cryptocurrency and equity markets suggests that the underlying stochastic processes are governed by heavy-tailed dynamics and asymptotic dependence.

EVT provides a non-parametric approach of modelling  heavy tailed observations where an obvious issue is to estimate the corresponding tail index which characterises the decay rate of the approximating power law like structure in the tail. Notable works in this area include \cite{geluk1996tails}, \cite{resnick1998tail}, \cite{mikosch2000supremum} and others. The concept of tail adversarial stability measure in a triangular array setting was first introduced by {\cite{zhang2021high}}. The tail adversarial q-stability condition concerns only about the upper tail of the distribution when compared to the strong mixing condition of \cite{rosenblatt1956central} and functional dependence framework of \cite{wu2005nonlinear}. It has been further documented in \cite{zhang2021high} that  q-stability mixing condition directly leads to the limit theorems in high quantile regression. 

The general problem of tackling correlated extreme events using the copula approach goes back to \cite{caperaa1997nonparametric}, \cite{zeevi2002beyond}, \cite{jajuga2005extreme}, \cite{demarta2005t} and others.
In this article, a new probabilistic perspective of defining tail cross correlation is addressed. We also emphasise on the two fundamental properties of linear processes, namely, one on the data generating mechanism and the other on the persistence. A second contribution is that we have formulated the theory based on both the iid setup and arbitrary mixture of regularly varying distributions with different tail indices. We also show that, regardless of the setup, the conditions for short memory vs. long memory persistence remains the same.

The rest of the paper is organized as follows. Section \ref{sec2} provides a basic setup for the rest of the theory. Section \ref{sec3} provides insight into the non-identical setup of regularly varying random variables whereas in Section \ref{TailCC}, our new measure of dependence have been proposed and different theorems have been stated. In both the sections, statistical analysis have been carried out using real data.
Section \ref{simulation} provides a detailed simulation study for both non-identical and identical cases. Proof of all lemmas and theorems are collected in an appendix.

\section{Baseline Result for iid setup}\label{sec2}

In the classical setup, which we first recount below, the tail properties are studied in an iid setting. We start with a few relevant definitions.

Let $(a_j)_{j \ge 0}, (b_j)_{j \ge 0}$ be two sequence of real coefficients.
Consider the additive linear processes 
\begin{align}
X_i = \sum_{j=0}^{\infty} a_j \epsilon^{(1)}_{i-j} \quad \textbf{ and } \quad
Y_i = \sum_{j=0}^{\infty} b_j \epsilon^{(2)}_{i-j}, \quad i \in \mathbb{Z}
\end{align}
where $(\epsilon^{(1)}_j)_{j \in \mathbb{Z}}$ and $(\epsilon^{(2)}_j)_{j \in \mathbb{Z}}$ are two independent sequences of i.i.d. innovation random variables.\\
Let $\epsilon^*_0$ be an innovation that has the same distribution as $\epsilon^{(1)}_0$, but is independent of $(\epsilon^{(1)}_j)_{j \in \mathbb{Z}}$ and $(\epsilon^{(2)}_j)_{j \in \mathbb{Z}}$.\\
Let $X_i^*,Y_i^*$ represent the coupled versions at time $i$ whose innovation at time zero is replaced by an i.i.d. copy $\epsilon^*_0$.
\begin{align}
X_i^* = \sum_{j=0, j \neq i}^{\infty} a_j \epsilon^{(1)}_{i-j} + a_i \epsilon^*_0, \quad \textbf{and} \quad
Y_i^* = \sum_{j=0, j \neq i}^{\infty} b_j \epsilon^{(2)}_{i-j} + b_i \epsilon^*_0, \quad i \in \mathbb{Z}
\end{align}

\begin{definition}
    A real-valued random variable $X$, or its distribution $F$, is said
to have \textbf{balanced regularly varying tails} with exponent $\alpha \ge 0$ 
if the nonnegative random variable $|X|$ is regularly varying with 
exponent $\alpha$ and for some $0 \le p, q \le 1$ with $p + q = 1$:
\[
\lim_{x \to \infty} \frac{P(X > x)}{P(|X| > x)} = p, \quad 
\lim_{x \to \infty} \frac{P(X < -x)}{P(|X| > x)} = q
\]
\end{definition}

\begin{definition}
    A positive Lebesgue measurable real-valued function L is said to be slowly varying at $\infty$ if 
    \[
    L(\lambda x) \sim L(x) \quad as \quad x \rightarrow \infty \quad \text{for each} \quad \lambda > 0.
    \]
\end{definition}
The following result acts as our benchmark.

\begin{theorem}[\cite{samorodnitsky2016stochastic}]
Let $\{X_n\}_{n \in \mathbb{Z}}$ be a sequence of i.i.d. random variables with balanced regularly varying tails with exponent $\alpha > 0$, where $P(X > x) \sim p P(|X| > x)$ and $P(X < -x) \sim q P(|X| > x)$ as $x \to \infty$. Let $(\rho_n)_{n \in \mathbb{Z}}$ be a sequence of real numbers satisfying:
\[
\begin{cases} 
\sum_{n=-\infty}^{\infty} |\rho_n|^{\alpha-\epsilon} < \infty & \text{for some } 0 < \epsilon < \alpha, \text{ if } \alpha \leq 2, \\
\sum_{n=-\infty}^{\infty} |\rho_n|^2 < \infty & \text{if } \alpha > 2.
\end{cases}
\]
Furthermore, if $\alpha > 1$, assume that either $E[X_0] = 0$ or the series $\sum_{n=-\infty}^{\infty} \rho_n$ converges. Then, the series 
\[
S = \sum_{n=-\infty}^{\infty} \rho_n X_n
\]
converges with probability 1. Moreover, $S$ has balanced regularly varying tails with exponent $\alpha$, satisfying:
\begin{align}
\lim_{x \to \infty} \frac{P(S > x)}{P(|X| > x)} &= \sum_{n=-\infty}^{\infty} \left[ p(\rho_n)_{+}^{\alpha} + q(\rho_n)_{-}^{\alpha} \right], \\
\lim_{x \to \infty} \frac{P(|S| > x)}{P(|X| > x)} &= \sum_{n=-\infty}^{\infty} |\rho_n|^{\alpha}.
\end{align}
\end{theorem}

\section{Non identical situation: empirical illustrations and extended results} {\label{sec3}}

The above discussion pertains to the iid setup. As mentioned before, our primary aim in this paper is to establish/extend analogous results in the finite mixture situation. Before we dive into the theoretical discussion, we first empirically verify whether such a situation is pertinent in real data settings. For this purpose, the statistical analysis presented in this article takes into consideration the closing prices of three highly traded cryptocurrencies namely Bitcoin (BTC), Ethereum (ETH) and Solana (SOL). High frequency data on the scale of 1 miliseconds is considered for the purpose of investigation. This data is publicly available at  \url{https://coincodex.com/crypto/bitcoin/historical-data/}. A basic histogram (see Figure \ref{fig2}) of the datasets shows bimodality. To gain insight to this visual representation, we have performed Hartigan's Dip test (test of unimodality vs bimodality \cite{hartigan1985dip}) and it supports our observation.
It may be suggestive of a possible mixture distribution with specific mixing probabilities.  These dataset exhibit changes in their distribution function, hence they are not iid. A rigorous analysis of deciseconds level data  of length 2306 for each of the three cryptocurrencies has been provided. Table \ref{Table51} identifies the breaks in the underlying structure, typically the distribution function of the three series respectively. Considering heavy tailed situations, three prominent members of this class are considered for the model fitting purpose namely, Pareto, Weibull and Cauchy. The goodness of fit has been assessed by AIC and BIC, as shown in columns 2-7 of Tables \ref{Table52}, \ref{Table53} and \ref{Table54} for the three cryptocurrencies considered. The best fitted models in different segments have been put together in column 8 of the same tables. \\
 
 Some typical naming conventions are as follows:
 \begin{flalign*}
\mathscr{P}(\alpha, \beta) &: \text{Pareto distribution with shape parameter } \alpha \text{ and scale } \beta &\\
\mathscr{C}(\theta, \gamma) &: \text{Cauchy distribution with location } \theta \text{ and scale } \gamma &\\
\mathscr{W}(k, \lambda)    &: \text{Weibull distribution with shape } k \text{ and scale } \lambda &\\
\mathscr{F}(\alpha, \beta, \sigma) &: \text{Fr\'{e}chet distribution with location } \alpha, \text{shape } \beta \text{ and scale } \sigma
\end{flalign*}

\begin{figure}[H]
\centering
\includegraphics[width=1\textwidth]{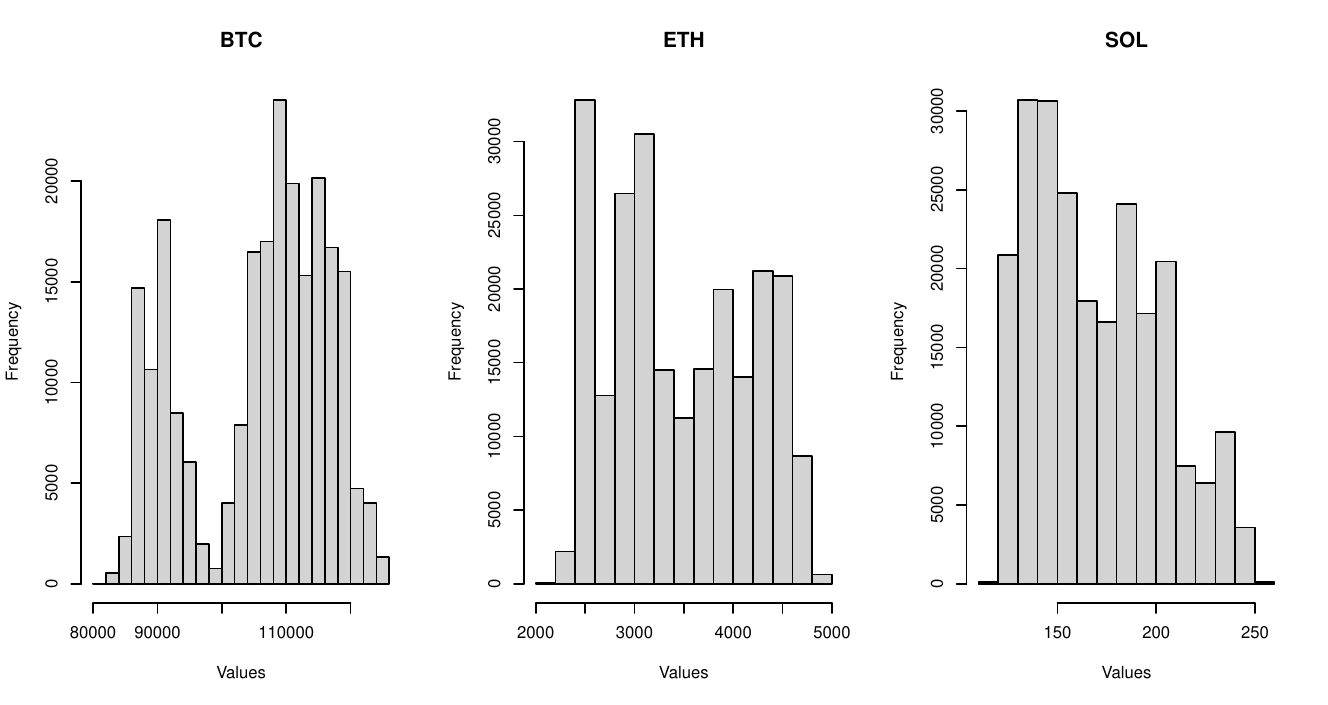}
\caption{Histogram of the three different cryptocurrencies BTC, ETH and SOL}
\label{fig2}
\end{figure}

\begin{center}
\small
\begin{longtable}{cccc}
\caption{Evidence of Non iid}{\label{Table51}}\\
\toprule
\textbf{} & \textbf{BTC} & \textbf{ETH} & \textbf{SOL} \\
\midrule
No of change points     & 5    & 5   & 5    \\
Locations     & (1673,463,1373,841,1973)   & (521,1581,821,1171,1951)    & (1663,876,1350,520,1963)    \\
\bottomrule
\end{longtable}
\end{center}

\begin{longtable}{l cc cc cc l}
\caption{Model Comparison with Robust and Standard Estimates BTC}{\label{Table52}} \\ 
\toprule
\multirow{2}{*}{Index} & \multicolumn{2}{c}{Pareto} & \multicolumn{2}{c}{Cauchy} & \multicolumn{2}{c}{Weibull} & \multirow{2}{*}{Best Fitted Model} \\
\cmidrule(lr){2-3} \cmidrule(lr){4-5} \cmidrule(lr){6-7} 
& AIC & BIC & AIC & BIC & AIC & BIC & \\
\midrule
\endfirsthead

\toprule
\multirow{2}{*}{Index} & \multicolumn{2}{c}{Pareto} & \multicolumn{2}{c}{Cauchy} & \multicolumn{2}{c}{Weibull} & \multirow{2}{*}{New Column} \\
\cmidrule(lr){2-3} \cmidrule(lr){4-5} \cmidrule(lr){6-7}
& AIC & BIC & AIC & BIC & AIC & BIC & \\
\midrule
\endhead

\bottomrule
\endfoot

1-463 & 11965.43 & 11969.57 & 8662.60 & 8670.87 & 8458.71* & 8466.99** & $\mathscr{W}(53.44,107687.61)$\\
464-841 & 9838.372 & 9842.306 & 6874.77* & 6882.64** & 6884.47 & 6892.34 & $\mathscr{C}(117941.67,917.50)$ \\
842-1373 & 13819.73 & 13824.01 & 10503.54* & 10512.1** & 10539 & 10547.56 &  $\mathscr{C}(113261.50)$ \\
1374-1673 & 7764.77 & 7768.47 & 5991.55 & 5998.95 & 5849* & 5856.40** & $\mathscr{W}(30.20,110100.62)$ \\
1674-1973 & 7658.62 & 7662.32 & 5769.57 & 5776.97 & 5746.15* & 5753.56** & $\mathscr{W}(28.30,92138.16)$ \\
1974-2306 & 8491.32 & 8495.13 & 6376.71 & 6384.33 & 6346.00* & 6353.62** & $\mathscr{W}(29.41,91581.76)$ \\

\bottomrule
\end{longtable}

\begin{longtable}{l cc cc cc l}
\caption{Model Comparison with Robust and Standard Estimates ETH}{\label{Table53}}  \\ 
\toprule
\multirow{2}{*}{Index} & \multicolumn{2}{c}{Pareto} & \multicolumn{2}{c}{Cauchy} & \multicolumn{2}{c}{Weibull}  & \multirow{2}{*}{Best Fitted Model}  \\
\cmidrule(lr){2-3} \cmidrule(lr){4-5} \cmidrule(lr){6-7}
& AIC & BIC & AIC & BIC & AIC & BIC &\\
\midrule
\endfirsthead 

\toprule
\multirow{2}{*}{Index} & \multicolumn{2}{c}{Column 2} & \multicolumn{2}{c}{Column 3} & \multicolumn{2}{c}{Column 4} \\
\cmidrule(lr){2-3} \cmidrule(lr){4-5} \cmidrule(lr){6-7}
& Sub A & Sub B & Sub A & Sub B & Sub A & Sub B \\
\midrule
\endhead 

\bottomrule
\endfoot 

1-521 & 9611.80 & 9616.05 & 6800.20* & 6808.71** & 7088.45 & 7096.96 &$\mathscr{C}(2532.55,64.61)$ \\
522-821 & 5773.24 & 5776.94 & 4390.30* & 4397.70** & 4467.58 & 4474.98 &$\mathscr{C}(3712.22,137.93)$\\
822-1171 & 6822.87 & 6826.73 & 4660.46 & 4668.18 & 4598.26* & 4605.97** & $\mathscr{W}(27.85,4525.77)$\\
1172-1581 & 7938.18 & 7942.20 & 5808.26 & 5816.29 & 5769.07* & 5777.10** &$\mathscr{W}(16.15,4209.42)$\\
1582-1951 & 6965.85 & 6969.76 & 5209.36 & 5217.19 & 5096.31* & 5104.13** &$\mathscr{W}(14.80,3240.72)$\\
1952-2306 & 6965.85 & 6969.76 & 4662.51 & 4670.25 & 4555.84* & 4563.58** &$\mathscr{W}(22.87,3138.62)$\\

\bottomrule
\end{longtable}

\begin{longtable}{l cc cc cc l}
\caption{Model Comparison with Robust and Standard Estimates SOL}{\label{Table54}}  \\ 
\toprule
\multirow{2}{*}{Index} & \multicolumn{2}{c}{Pareto} & \multicolumn{2}{c}{Cauchy} & \multicolumn{2}{c}{Weibull}  & \multirow{2}{*}{Best Fitted Model} \\
\cmidrule(lr){2-3} \cmidrule(lr){4-5} \cmidrule(lr){6-7}
& AIC & BIC & AIC & BIC & AIC & BIC & \\
\midrule
\endfirsthead 

\toprule
\multirow{2}{*}{Index} & \multicolumn{2}{c}{Column 2} & \multicolumn{2}{c}{Column 3} & \multicolumn{2}{c}{Column 4} \\
\cmidrule(lr){2-3} \cmidrule(lr){4-5} \cmidrule(lr){6-7}
& Sub A & Sub B & Sub A & Sub B & Sub A & Sub B \\
\midrule
\endhead 

\bottomrule
\endfoot 

1-520 & 6663.89 &  6668.15 & 4049.67 & 4058.18 & 4003.90* & 4012.41** &$\mathscr{W}(14.78,159.51)$\\
521-876 & 4671.91 & 4675.78 & 2804.70 & 2812.45 & 2729.18* & 2736.93** &$\mathscr{W}(18.05,186.24)$\\
877-1350 & 6393.60 & 6397.77 & 4202.53 & 4210.86 & 3989.45* & 3997.77** &$\mathscr{W}(15.05,224.88)$\\
1351-1663 & 4115.61 & 4119.36 & 2688.90 & 2696.39 & 2554.44* & 2561.93** &$\mathscr{W}(15.33,188.86)$\\
1664-1963 & 3761.83 & 3765.54 & 2012.20 & 2019.61 & 1976.68* & 1984.09** &$\mathscr{W}(22.30,139.52)$\\
1964-2306 & 4274.96 & 4278.80 & 2579.17 & 2586.84 & 2415.83* & 2423.50** &$\mathscr{W}(18.11,135.27)$\\

\bottomrule
\footnotetext{* Indicates the lowest AIC (best fit); ** Indicates the lowest BIC (best fit).}
\end{longtable}
\footnotetext{* Indicates the lowest AIC (best fit); ** Indicates the lowest BIC (best fit).}
Inspired by this empirical evidence, it becomes imperative to look into the mixture distributions, specifically at the tail probabilities.
Detection of tail index for finite mixture of regularly varying distributions with different indices has been an instance of prior importance. One of the ways to identify the tail index is to look at the logarithm of the survival function beyond a certain threshold. Theorem \ref{Thm} provides an insight on how these indices can be traced back from the observed data. Part (a) can be looked as an application of (\cite{embrechts2013modelling} Lemma 1.3.1).The bound provided in (c) takes into account the second order variation of the indices, an instance stronger than just looking at the first order variation. Bounds in (b) and (c) are non-comparable but they stand out in their individual importance.

\begin{theorem}{\label{Thm}}
Let $\mathbf{X} = (X_1, \dots, X_n)'$ be a vector of independent heavy-tailed random variables. The following asymptotic bounds hold for the log-tail probability of the linear combination $\mathbf{l}'\mathbf{X}$:

\begin{enumerate}
    \item[(a)] There exists $x^* \in \mathbb{R}$ such that for all $x \geq x^*$:
\begin{align*}
\log \mathbb{P}(\mathbf{l}'\mathbf{X} > x) \sim & -\alpha \left( \log\left(\frac{x}{l_*}\right)\mathbb{I}(l_*>0) + \log\left(\frac{x}{|l_*|}\right)\mathbb{I}(l_*<0) \right) + \\
& (\log(p) + \log L(x/l_*)) \mathbb{I}(l_* > 0) + \\
& (\log(q) + \log L(x/|l_*|)) \mathbb{I}(l_* < 0)
\end{align*}
    where $\alpha = \min_{j} \{\alpha_{j}\}$ and $l_*$ is the coefficient associated with the variable $X_i$ possessing the minimal tail index.

    \item[(b)] The log-tail probability is bounded by the summation of marginal indices:
    \[
    \log \mathbb{P}(\mathbf{l}'\mathbf{X} > x) \leq -\sum_{i=1}^{n} \alpha_{i} \log\left(\frac{x}{n l^*}\right) + \log L\left(\frac{x}{n l^*}\right)
    \]
    where $l^* = \min_i |l_i|$.

    \item[(c)] The log-tail probability is bounded by the moments of marginal indices::
    \[
    \log \mathbb{P}(\mathbf{l}'\mathbf{X} > x) \leq -\left( \frac{ \sum_{i=1}^{n}\alpha_i - \frac{1}{2}\sum_{i=1}^{n}\alpha_i^2 }{n} \right) \log\left(\sqrt{\frac{x^2}{\sum_{i=1}^{n}l_i^2}}\right) + \log(n) + \log L(x)
    \]
\end{enumerate}
\end{theorem}

To illustrate the practical efficacy of the bounds established in Theorem \ref{Thm}, consider the following two examples: \\
\textbf{(i)} Let U $\sim$ $\mathscr{P}(2.414,1)$ and V $\sim$ $\mathscr{C}(0,1)$ independently. Define $W$=$U$+$\frac{1}{3}V$, then the estimate of slope by Theorem \ref{Thm}(a) and Theorem \ref{Thm}(b)  is -1 and -3.414 respectively.
Thus, Theorem \ref{Thm}(b)  has sharper slope bound than Theorem \ref{Thm}(a) .\\
\textbf{(ii)} Let R $\sim$ $\mathscr{P}(5,1)$ and define $S$=$U$+$\frac{1}{3}R$,, then the estimate of slope by Theorem \ref{Thm}(a)  and Theorem \ref{Thm}(c)  is -2.414 and -3.99 respectively. Thus, Theorem \ref{Thm}(c)  has sharper slope bound than Theorem \ref{Thm}(c).

\section{Tail cross-correlation and memory}\label{TailCC}

\subsection{Theory}
\label{TailCCTheory}

Having established the mixing possibility and its consequences in terms of tail properties, we now move on to the exceedance clustering issue. Our objective in this section is to define a measure of exceedance clustering at high thresholds and to study its memory retention properties at large lags.

Let $x_n, y_n \to \infty$ be an extremal threshold for the perturbed sequences $\{X_t^*\}$ and $\{Y_t^*\}$ respectively. The tail cross-correlation coefficient at lag $k$ is defined as:

\begin{equation}
\tau_{x_n, y_n}(k) = \frac{P(Y^{*}_{t+k} > y_n \mid X^{*}_{t} > x_n) - P(Y^{*}_{t+k} > y_n)}{1 - P(X^{*}_{t} > x_n)}
\end{equation}

We shall explore the persistence property under both the iid case and arbitrary mixture of independent distributions. Theorem \ref{thm:IID1} and Theorem \ref{thm:IID2} illustrates iid situation having the tail index $\alpha_1$ and $\alpha_2$ for $X^*$ and $Y^*$ series respectively whereas Theorem \ref{thm:NID1} and Theorem \ref{thm:NID2} does the non-iid part having $\alpha_0$ and $\alpha_i^{(2)}$ for $X^*$ and $Y^*$ series respectively.\\
For univariate observations, there exists a notion of ordering typically given by the quantiles whereas for high dimensional data, no such ordering exists and hence one relies on depth functions. \\
Consider, the following scenario where $U$ and $V$ are random variables with corresponding quantile sequences $\{u_n\}$ and $\{v_n\}$ respectively.      
If $U$ is much heavier-tailed than $V$, then for a large $n$, $\mathbb{P}(U > u_n) \gg \mathbb{P}(Y > u_n)$. The joint probability would be dominated entirely by the lighter tail of $V$, potentially masking the true relationship. By allowing $\{u_n\} \neq \{v_n\}$, one can calibrate the thresholds so that you are looking at the \enquote{top 1} of $U$ and the \enquote{top 1}" of $V$ simultaneously.
Since, in our setup $\{x_n\}_{n=1}^{\infty}$ and  $\{y_n\}_{n=1}^{\infty}$ are typically quantile sequences for $X^*$ and $Y^*$ respectively, we shall formulate the theorems initially on identical quantile sequences $\{x_n\}_{n=1}^{\infty}$ = $\{y_n\}_{n=1}^{\infty}$ and later on $\{x_n\}_{n=1}^{\infty}$ $\neq$ $\{y_n\}_{n=1}^{\infty}$. \\

We shall explore the persistence property under both the conditions, namely iid setup and non-iid setup. We need to focus on situations that exhibit short (long) term memory. Hence, the following lemma.
\begin{lemma}\label{mem.lemma}
The underlying process exhibits distinct memory structures depending on the decay rate of the coefficients $\{b_j\}$:
\begin{enumerate}
    \item \textbf{Short-Term Memory:} The process has short-range dependence if the coefficients decay exponentially:
    \begin{equation}
        b_j = \phi^j, \quad \text{where } |\phi| < 1.
    \end{equation}
    
    \item \textbf{Long-Term Memory:} The process exhibits long-range dependence (power-law decay) if:
    \begin{equation}
        b_j = j^{-\beta}, \quad \text{where } \beta > 0.
    \end{equation}
    
\end{enumerate}
\end{lemma}

With the above characteristic coefficient structures, we now state the main results of our paper in four theorems representing the four possible combinations, as discussed above.

\begin{theorem}\textbf{(IID, Identical Quantile)} \label{thm:IID1}
\\
The tail cross-correlation coefficient $\tau(.)$ satisfies the following properties:

\begin{enumerate}
    \item \textbf{Monotonicity Property:} 
    Under the assumptions of Theorem 1, the sequence $\tau(.)$ is non-increasing, $\tau_{(i+1)} \leq \tau_{(i)}$, provided that $ \frac{|b_{i+2}|}{| b_{i+1}|} < 1$ and exactly one of the following conditions on the tail indices  and slowly varying functions $L(\cdot)$ holds:
    \begin{enumerate}
        \item $\alpha_2 > \alpha_1$
        \item $\frac{L_2(x)}{L_1(x)} = o\left(x^{-(\alpha_2 - \alpha_1)}\right) \quad \text{as } x \to \infty$
    \end{enumerate}

    \item \textbf{Asymptotic Ratio:} 
   Under the above mentioned conditions, the ratio of successive tail cross-correlations satisfies:
    \begin{equation}
        \frac{\tau_{(i+1)}}{\tau_{(i)}} \approx \left| \frac{b_{i+2}}{b_{i+1}} \right|^{\alpha_{1}}
    \end{equation}
\end{enumerate}
\end{theorem}

\begin{theorem}\textbf{(Non-IID, Identical Quantile)}\label{thm:NID1}
\\
The tail cross-correlation coefficient $\tau(.)$ satisfies the following properties:

\begin{enumerate}
    \item \textbf{Monotonicity Property:} 
    The sequence is non-increasing, $\tau_{(i+1)} \leq \tau_{(i)}$, provided that $\frac{|b_{i+2}|^{\alpha_0}-|b_{i+1}|^{\alpha_0}}{[p(a_1)_+^{\alpha_0}+q(a_1)_-^{\alpha_0}]}<1$ and exactly one of the following conditions on the tail indices $\alpha$ and slowly varying functions $L(\cdot)$ holds:
    \begin{enumerate}
        \item $ \alpha_i^{(2)}> \alpha_0 $ $\forall i$
        \item $\frac{L_i(x)}{L(x)} = o\left(x^{-(\alpha_i^{(2)} - \alpha_0)}\right) \quad \text{as } x \to \infty$
    \end{enumerate}

    \item \textbf{Asymptotic Ratio:} 
    Under the above mentioned conditions, the ratio of successive tail cross-correlations satisfies:
    \begin{equation}
        \frac{\tau_{(i+1)}}{\tau_{(i)}} \approx \left| \frac{b_{i+2}}{b_{i+1}} \right|^{\alpha_{0}}
    \end{equation}
\end{enumerate}
\end{theorem}

Here, we shall present the theorems based on the generalization under $\{x_n\}$ and $\{y_n\}$. Also observe that the memory property remains the same.

\begin{theorem}\textbf{(IID, Non-Identical Quantile)}\label{thm:IID2}
\\
The tail cross-correlation coefficient $\tau(.)$ satisfies the following properties:

\begin{enumerate}
    \item \textbf{Monotonicity Property:} 
    The sequence is non-increasing, $\tau_{(i+1)} \leq \tau_{(i)}$, provided that $\frac{|b_{i+2}|^{\alpha_1}-|b_{i+1}|^{\alpha_1}}{\sum_{j=0}^{\infty} [p(b_j)_+^{\alpha_1}+q(b_j)_-^{\alpha_1}]}<\left( \frac{x}{y} \right)^{-\alpha_1}\frac{L(x)}{L(y)}$ and exactly one of the following conditions on the tail indices $\alpha$ and slowly varying functions $L(\cdot)$ holds:
 \begin{enumerate}
    \item $y^{-\alpha_2} =o\left( x^{-\alpha_1}\right)$ 
    \item $\frac{L_1(y)}{L(x)} = o\left( \frac{y^{-\alpha_2}}{x^{-\alpha_1}} \right) \quad \text{as } x,y \to \infty$
\end{enumerate}

    \item \textbf{Asymptotic Ratio:} 
    Under the conditions for similar sequences, the ratio of successive tail cross-correlations satisfies:
    \begin{equation}
        \frac{\tau_{(i+1)}}{\tau_{(i)}} \approx \left| \frac{b_{i+2}}{b_{i+1}} \right|^{\alpha_{1}}
    \end{equation}
\end{enumerate}
\end{theorem}

\begin{theorem}\textbf{(Non-IID, Non-Identical Quantile)}\label{thm:NID2}
\\
The tail cross-correlation coefficient $\tau(.)$ satisfies the following properties:

\begin{enumerate}
    \item \textbf{Monotonicity Property:} 
    The sequence is non-increasing, $\tau_{(i+1)} \leq \tau_{(i)}$, provided that $\frac{|b_{i+2}|^{\alpha_0}-|b_{i+1}|^{\alpha_0}}{[p(a_1)_+^{\alpha_0}+q(a_1)_-^{\alpha_0}]}<\left( \frac{x}{y} \right)^{-\alpha_0}\frac{L(x)}{L(y)}$ and exactly one of the following conditions on the tail indices $\alpha$ and slowly varying functions $L(\cdot)$ holds:
 \begin{enumerate}
    \item $y^{-\alpha_i^{(2)}} =o \left( x^{-\alpha_0}\right) \quad \forall i$
    \item $\frac{L_i(y)}{L(x)} = o\left( \frac{y^{-\alpha_i^{(2)}}}{x^{-\alpha_0}} \right) \quad \text{as } x,y \to \infty$
\end{enumerate}

    \item \textbf{Asymptotic Ratio:} 
    Under the conditions for similar sequences, the ratio of successive tail cross-correlations satisfies:
    \begin{equation}
        \frac{\tau_{(i+1)}}{\tau_{(i)}} \approx \left| \frac{b_{i+2}}{b_{i+1}} \right|^{\alpha_{0}}
    \end{equation}
\end{enumerate}
\end{theorem}

To generalize in detail the interpretation of these theorems, we note the following. The monotonicity property of the tail cross-correlation coefficient depends on two factors, namely tail index of the random innovations and the regularly varying function at $\infty$. The fraction of the tail cross correlation coefficient under any setup, iid as well as non-iid and under any quantile both identical and non-identical only depends on the tail index of the perturbing random innovation. It implies a reduction of complexity. Essentially, we are arguing that while the short-term dynamics of a functional time series can be incredibly messy, the long-term risk (the tails) is governed by a surprisingly small set of parameters.

\subsection{Illustrations}

To empirically verify the salience of our theoretical finding on the tail cross correlation, we now reach back to the cryptocurrency data sets used above.
We put forward the observations from our statistical analysis.
In the study of high-frequency financial time series and functional data, the interaction between persistence (Long Memory) and extreme co-movements (Tail Cross-Correlation) is fundamental to understanding market stability and contagion. Table \ref{Table6} provides estimates of hurst exponent and  Geweke and Porter-Hudak (GPH) estimate which shows that these three series have long memory property. The same has been verified using our proposed tail cross correlation coefficient in Table \ref{tab:tail_cross_corr_combined}. It is observed that the correlation values do not tail off even at very distant lags.
\begin{center}
\small
\begin{longtable}{cccc}
\caption{Evidence of Long Memory}{\label{Table6}}\\
\toprule
\textbf{} & \textbf{BTC} & \textbf{ETH} & \textbf{SOL} \\
\midrule
Hurst Exponent     & 0.927    & 0.934     & 0.930     \\
GPH Estimate     & 0.987    & 1.011    & 0.994     \\
\bottomrule
\end{longtable}
\end{center}
\begin{table}[htbp]
\centering
\caption{Tail Cross-Correlation Coefficients for BTC, ETH, and SOL at Multiple Lags}
\label{tab:tail_cross_corr_combined}
\footnotesize
\setlength{\tabcolsep}{3pt} 
\begin{tabular}{l ccc ccc ccc}
\toprule
& \multicolumn{3}{c}{\textbf{Lag 1}} & \multicolumn{3}{c}{\textbf{Lag 3}} & \multicolumn{3}{c}{\textbf{Lag 5}} \\
\cmidrule(lr){2-4} \cmidrule(lr){5-7} \cmidrule(lr){8-10}
\textbf{$(q1,q2)$} & \textbf{(BTC,ETH)} & \textbf{(ETH,SOL)} & \textbf{(BTC,SOL)} & \textbf{(BTC,ETH)} & \textbf{(ETH,SOL)} & \textbf{(BTC,SOL)} & \textbf{(BTC,ETH)} & \textbf{(ETH,SOL)} & \textbf{(BTC,SOL)} \\
\midrule
(0.75,0.75) & 0.422 & 0.737 & 0.298 & 0.421 & 0.737 & 0.297 & 0.421 & 0.737 & 0.297 \\
(0.75,0.85) & 0.363 & 0.563 & 0.266 & 0.362 & 0.563 & 0.266 & 0.362 & 0.563 & 0.266 \\
(0.75,0.95) & 0.153 & 0.200 & 0.196 & 0.153 & 0.200 & 0.196 & 0.153 & 0.200 & 0.196 \\
(0.85,0.75) & 0.275 & 0.753 & 0.159 & 0.275 & 0.753 & 0.159 & 0.275 & 0.753 & 0.159 \\
(0.85,0.85) & 0.265 & 0.611 & 0.154 & 0.265 & 0.611 & 0.154 & 0.265 & 0.611 & 0.154 \\
(0.85,0.95) & 0.103 & 0.327 & 0.018 & 0.103 & 0.327 & 0.018 & 0.102 & 0.327 & 0.018 \\
(0.95,0.75) & 0.637 & 0.743 & 0.557 & 0.637 & 0.743 & 0.557 & 0.637 & 0.743 & 0.557 \\
(0.95,0.85) & 0.587 & 0.408 & 0.565 & 0.587 & 0.408 & 0.565 & 0.587 & 0.408 & 0.565 \\
(0.95,0.95) & 0.219 & 0.339 & 0.098 & 0.219 & 0.338 & 0.098 & 0.220 & 0.338 & 0.098 \\
\midrule \midrule
& \multicolumn{3}{c}{\textbf{Lag 7}} & \multicolumn{3}{c}{\textbf{Lag 30}} & \multicolumn{3}{c}{\textbf{Lag 100}} \\
\cmidrule(lr){2-4} \cmidrule(lr){5-7} \cmidrule(lr){8-10}
\textbf{$(q1,q2)$} & \textbf{(BTC,ETH)} & \textbf{(ETH,SOL)} & \textbf{(BTC,SOL)} & \textbf{(BTC,ETH)} & \textbf{(ETH,SOL)} & \textbf{(BTC,SOL)} & \textbf{(BTC,ETH)} & \textbf{(ETH,SOL)} & \textbf{(B,S)} \\
\midrule
(0.75,0.75) & 0.421 & 0.737 & 0.297 & 0.419 & 0.736 & 0.297 & 0.413 & 0.736 & 0.293 \\
(0.75,0.85) & 0.362 & 0.563 & 0.266 & 0.362 & 0.564 & 0.265 & 0.360 & 0.564 & 0.263 \\
(0.75,0.95) & 0.153 & 0.200 & 0.196 & 0.154 & 0.200 & 0.196 & 0.154 & 0.200 & 0.196 \\
(0.85,0.75) & 0.274 & 0.753 & 0.159 & 0.274 & 0.754 & 0.159 & 0.271 & 0.754 & 0.159 \\
(0.85,0.85) & 0.265 & 0.611 & 0.154 & 0.265 & 0.609 & 0.154 & 0.265 & 0.609 & 0.153 \\
(0.85,0.95) & 0.102 & 0.327 & 0.018 & 0.102 & 0.328 & 0.018 & 0.101 & 0.328 & 0.018 \\
(0.95,0.75) & 0.637 & 0.743 & 0.557 & 0.635 & 0.742 & 0.557 & 0.629 & 0.742 & 0.560 \\
(0.95,0.85) & 0.587 & 0.408 & 0.565 & 0.585 & 0.410 & 0.565 & 0.578 & 0.410 & 0.565 \\
(0.95,0.95) & 0.220 & 0.338 & 0.097 & 0.219 & 0.335 & 0.097 & 0.216 & 0.335 & 0.094 \\
\bottomrule
\end{tabular}
\end{table}

\section{Simulation Studies for Mixture Situations} \label{simulation}

To highlight the robustness of our findings to alternative probability models, we undertake a detailed simulation study in this section. We have considered\[
Y_i^* = \sum_{j = i-5}^{i-1} b_j \epsilon^{(2)}_{i-j} + b_i \epsilon^*_0, \quad i \in \{1,2,3,4,5,6\}
\] along with the same notation as used in Section \ref{sec3}.
Table \ref{Table1} and \ref{Table2} provide simulations under iid setup. The innovations in the four parts of each of these are generated from $\mathscr{P}(3,1)$, $\mathscr{P}(10,1)$, $\mathscr{C}(0,1)$ and $\mathscr{C}(0,10)$ distributions respectively. The four different choices of perturbations considered here are $\mathscr{P}(2.414,1)$, $\mathscr{W}(1,1)$, $\mathscr{W}(0.5,1)$ and $\mathscr{F}(0,1,1)$.

Tables \ref{Table3} and \ref{Table4} show simulations in the independent (but not identical) setup. In both the tables wherever $\mathscr{P}(.,1)$ is mentioned, it means $\epsilon^{(2)}_{i}$ $\sim$ $\mathscr{P}(5-i+1,1)$ distribution and when
$\mathscr{C}(0,.)$ is mentioned, it means $\epsilon^{(2)}_{i}$ $\sim$ $\mathscr{C}(0,5-i+1)$ for i $\in$ $\{1,\ldots ,5\}$ distribution. The same perturbations have been followed as in Table \ref{Table1} and \ref{Table2}. \\

\begin{center}
\footnotesize
\setlength{\tabcolsep}{3pt}
\begin{longtable}{cc cccc cccc}
\caption{Short Term Memory Comparison using values of Hurst exponent}{\label{Table1}}\\
    \toprule
    & & \multicolumn{4}{c}{\textbf{$\mathscr{P}(3,1)$}} & \multicolumn{4}{c}{\textbf{$\mathscr{P}(10,1)$}} \\ 
    \cmidrule(lr){3-6} \cmidrule(lr){7-10}
    $\phi$ & $i$ & $\mathscr{P}(2.414,1)$ & $\mathscr{W}(1,1)$ & $\mathscr{W}{(.5,1)}$ & $\mathscr{F}(0,1,1)$ & $\mathscr{P}(2.414,1)$ & $\mathscr{W}(1,1)$ & $\mathscr{W}(.5,1)$ & $\mathscr{F}(0,1,1)$ \\
    \midrule
\endfirsthead

    \multicolumn{10}{c}{{\bfseries \tablename\ \thetable{} -- Continued from previous page}} \\
    \toprule
    & & \multicolumn{4}{c}{\textbf{$\mathscr{P}(3,1)$}} & \multicolumn{4}{c}{\textbf{$\mathscr{P}(10,1)$}} \\ 
    \cmidrule(lr){3-6} \cmidrule(lr){7-10}
    $\phi$ & $i$ & $\mathscr{P}(2.414,1)$ & $\mathscr{W}(1,1)$ & $\mathscr{W}(.5,1)$ & $\mathscr{F}(0,1,1)$ & $\mathscr{P}(2.414,1)$ & $\mathscr{W}(1,1)$ & $\mathscr{W}(.5,1)$ & $\mathscr{F}(0,1,1)$ \\
    \midrule
\endhead

    \bottomrule
\endfoot

    \multirow{6}{*}{0.1} 
    & 1 & 0.478 & 0.478     & 0.478     & 0.499     & 0.478 & 0.478     & 0.478     & 0.499     \\
   & 2 & 0.478 & 0.478     & 0.478     & 0.499     & 0.478 & 0.478     & 0.478     & 0.499   \\
   & 3 & 0.478 & 0.478     & 0.478     & 0.499     & 0.478 & 0.478     & 0.478     & 0.499    \\
   & 4 & 0.478 & 0.478     & 0.478     & 0.499     & 0.478 & 0.478     & 0.478     & 0.499     \\
   & 5 & 0.478 & 0.478     & 0.478     & 0.499     & 0.478 & 0.478     & 0.478     & 0.499 \\
   & 6 & 0.478 & 0.478     & 0.478     & 0.499     & 0.478 & 0.478     & 0.478     & 0.499\\ \midrule
    \multirow{6}{*}{0.25} 
    & 1 & 0.476 & 0.476     & 0.476     & 0.497     & 0.476 & 0.476     & 0.476     & 0.497    \\
    & 2 &  0.476 & 0.476     & 0.476     & 0.497   & 0.476 & 0.476     & 0.476     & 0.497    \\
    & 3 & 0.476 & 0.476     & 0.476     & 0.497     & 0.476 & 0.476     & 0.476     & 0.497     \\
    & 4 &  0.476 & 0.476     & 0.476     & 0.497   & 0.476 & 0.476     & 0.476     & 0.497    \\
    & 5 & 0.476 & 0.476     & 0.476     & 0.497 & 0.476 & 0.476     & 0.476     & 0.497 \\
    & 6 & 0.476 & 0.476     & 0.476     & 0.497& 0.476 & 0.476     & 0.476     & 0.497 \\ 
    \midrule
    \multirow{6}{*}{0.5} 
    & 1 & 0.466 & 0.465 & 0.465     & 0.510     &  0.465 & 0.465 & 0.465     & 0.510      \\
    & 2 &  0.466 & 0.465 & 0.465     & 0.510     &  0.465 & 0.465 & 0.465     & 0.510     \\
    & 3 & 0.466 & 0.465 & 0.465     & 0.510     & 0.465 & 0.465 & 0.465     & 0.510     \\
    & 4 &  0.466 & 0.465 & 0.465     & 0.510     & 0.465 & 0.465 & 0.465     & 0.510     \\
    & 5 &  0.466 & 0.465 & 0.465     & 0.510 & 0.465 & 0.465 & 0.465     & 0.510  \\
    & 6 &  0.466 & 0.465 & 0.465     & 0.510  & 0.465 & 0.465 & 0.465     & 0.510 \\ \midrule

    \pagebreak[0]
 & & \multicolumn{4}{c}{\textbf{$\mathscr{C}(0,1)$}} & \multicolumn{4}{c}{\textbf{$\mathscr{C}(0,10)$}} \\ 
    \cmidrule(lr){3-6} \cmidrule(lr){7-10}
    $\phi$ & $i$ & $\mathscr{P}(2.414,1)$ & $\mathscr{W}(1,1)$ & $\mathscr{W}{(.5,1)}$ & $\mathscr{F}(0,1,1)$ & $\mathscr{P}(2.414,1)$ & $\mathscr{W}(1,1)$ & $\mathscr{W}(.5,1)$ & $\mathscr{F}(0,1,1)$ \\
     \midrule
    \multirow{6}{*}{0.1} 
    & 1 & 0.525 & 0.525     & 0.525     & 0.532     & 0.525 & 0.525     & 0.525     & 0.532    \\
    & 2 & 0.525 & 0.525     & 0.525     & 0.532    & 0.525 & 0.525     & 0.525     & 0.532    \\
    & 3 & 0.525 & 0.525     & 0.525     & 0.532     & 0.525 & 0.525     & 0.525     & 0.532    \\
    & 4 & 0.525 & 0.525     & 0.525     & 0.532    & 0.525 & 0.525     & 0.525     & 0.532    \\
    & 5 & 0.525 & 0.525     & 0.525     & 0.532 & 0.525 & 0.525     & 0.525     & 0.532 \\
    & 6 & 0.525 & 0.525     & 0.525     & 0.532     &0.525 & 0.525     & 0.525     & 0.532 \\ \midrule
    \multirow{6}{*}{0.25} 
    & 1 & 0.512 & 0.512     & 0.512     & 0.543     & 0.512 & 0.512     & 0.512     & 0.543      \\
    & 2 & 0.512 & 0.512     & 0.512     & 0.543     & 0.512 & 0.512     & 0.512     & 0.543     \\
    & 3 & 0.512 & 0.512     & 0.512     & 0.543      & 0.512 & 0.512     & 0.512     & 0.543      \\
    & 4 & 0.512 & 0.512     & 0.512     & 0.543      & 0.512 & 0.512     & 0.512     & 0.543      \\
    & 5 & 0.512 & 0.512     & 0.512     & 0.543 & 0.512 & 0.512     & 0.512     & 0.543  \\
    & 6 & 0.512 & 0.512     & 0.512     & 0.543      & 0.512 & 0.512     & 0.512     & 0.543  \\ \midrule
    \multirow{6}{*}{0.5} 
    & 1 & 0.503 & 0.501 & 0.501     & 0.551     & 0.501 & 0.501     & 0.501     & 0.550     \\
    & 2 &  0.503 & 0.501 & 0.501     & 0.551     & 0.501 & 0.501     & 0.501     & 0.550    \\
    & 3 &  0.503 & 0.501 & 0.501     & 0.551      & 0.501 & 0.501     & 0.501     & 0.550     \\
    & 4 &  0.503 & 0.501 & 0.501     & 0.551     & 0.501 & 0.501     & 0.501     & 0.550     \\
    & 5 &  0.503 & 0.501 & 0.501     & 0.551  & 0.501 & 0.501     & 0.501     & 0.550 \\
    & 6 &  0.503 & 0.501 & 0.501     & 0.551     & 0.501 & 0.501     & 0.501     & 0.550 \\
\end{longtable}
\end{center}

\begin{center}
\footnotesize
\setlength{\tabcolsep}{2.5pt}
\begin{longtable}{cc cccc cccc}
\caption{Long Term Memory Comparison using the values of Hurst exponent} \label{Table2} \\
    \toprule
    & & \multicolumn{4}{c}{\textbf{$\mathscr{P}(3,1)$}} & \multicolumn{4}{c}{\textbf{$\mathscr{P}(10,1)$}} \\ 
    \cmidrule(lr){3-6} \cmidrule(lr){7-10}
    $\beta$ & $i$ & $\mathscr{P}(2.414,1)$ & $\mathscr{W}(1,1)$ & $\mathscr{W}(.5,1)$ & $\mathscr{F}(0,1,1)$ & $\mathscr{P}(2.414,1)$ & $\mathscr{W}(1,1)$ & $\mathscr{W}(.5,1)$ & $\mathscr{F}(0,1,1)$\\
    \midrule
\endfirsthead

    \toprule
    & & \multicolumn{4}{c}{\textbf{$\mathscr{C}(0,1)$}} & \multicolumn{4}{c}{\textbf{$\mathscr{C}(0,10)$}} \\ 
    \cmidrule(lr){3-6} \cmidrule(lr){7-10}
    $\beta$ & $i$ & $\mathscr{P}(2.414,1)$ & $\mathscr{W}(1,1)$ & $\mathscr{W}(.5,1)$ & $\mathscr{F}(0,1,1)$ & $\mathscr{P}(2.414,1)$ & $\mathscr{W}(1,1)$ & $\mathscr{W}(.5,1)$ & $\mathscr{F}(0,1,1)$ \\
    \midrule
\endhead

    \bottomrule
\endfoot

    \multirow{6}{*}{2} 
    & 1 & 0.521 & 0.505 & 0.505     & 0.510     & 0.507 & 0.505 & 0.505     & 0.511     \\
    & 2 & 0.517 & 0.519 & 0.519 & 0.513 & 0.518 & 0.519 & 0.519     & 0.513     \\
    & 3 & 0.498 & 0.497 & 0.497 & 0.499 & 0.496 & 0.497 & 0.497 & 0.500 \\
    & 4 & 0.475 & 0.468 & 0.468 & 0.506 & 0.470 & 0.468 & 0.468 & 0.506 \\
    & 5 & 0.482 & 0.486 & 0.486 & 0.503 & 0.485 & 0.486 & 0.486 & 0.503 \\
    & 6 & 0.471 & 0.471 & 0.471 & 0.470 & 0.471 & 0.471 & 0.471 & 0.497 \\ \midrule
    \multirow{6}{*}{10} 
    & 1 & 0.525 & 0.495 & 0.495 & 0.513 & 0.506 & 0.495 & 0.495 & 0.513 \\
    & 2 & 0.520 & 0.520 & 0.520 & 0.508 & 0.520 & 0.520 & 0.520 & 0.508 \\
    & 3 & 0.491 & 0.491 & 0.491 & 0.497 & 0.491 & 0.491 & 0.491 & 0.497 \\
    & 4 & 0.488 & 0.488 & 0.488     & 0.510     & 0.488 & 0.488 & 0.488 & 0.510 \\
    & 5 & 0.547 & 0.547 & 0.547 & 0.482 & 0.547 & 0.547 & 0.547 & 0.482 \\
    & 6 & 0.478 & 0.478 & 0.478 & 0.499 & 0.478 & 0.478 & 0.478 & 0.499 \\ \midrule
    \multirow{6}{*}{20} 
    & 1 & 0.525 & 0.495 & 0.495 & 0.513 & 0.506 & 0.495 & 0.495 & 0.513 \\
    & 2 & 0.520 & 0.520 & 0.520 & 0.508 & 0.520 & 0.520 & 0.520 & 0.508 \\
    & 3 & 0.491 & 0.491 & 0.491 & 0.497 & 0.491 & 0.491 & 0.491 & 0.497 \\
    & 4 & 0.489 & 0.489 & 0.488     & 0.510     & 0.489 & 0.489 & 0.489 & 0.510 \\
    & 5 & 0.547 & 0.547 & 0.547 & 0.482 & 0.547 & 0.547 & 0.547 & 0.482 \\
    & 6 & 0.478 & 0.478 & 0.478 & 0.499 & 0.478 & 0.478 & 0.478 & 0.499 \\ \midrule
   \pagebreak[0]
 & & \multicolumn{4}{c}{\textbf{$\mathscr{C}(0,1)$}} & \multicolumn{4}{c}{\textbf{$\mathscr{C}(0,10)$}} \\ 
    \cmidrule(lr){3-6} \cmidrule(lr){7-10}
    $\beta$ & $i$ & $\mathscr{P}(2.414,1)$ & $\mathscr{W}(1,1)$ & $\mathscr{W}(.5,1)$ & $\mathscr{F}(0,1,1)$ & $\mathscr{P}(2.414,1)$ & $\mathscr{W}(1,1)$ & $\mathscr{W}(.5,1)$ & $\mathscr{F}(0,1,1)$ \\
     \midrule

    \multirow{6}{*}{2} 
    & 1 & 0.525 & 0.537 & 0.538 & 0.504 & 0.551 & 0.537 & 0.537 & 0.512 \\
    & 2 & 0.524 & 0.524 & 0.524 & 0.494 & 0.524 & 0.524 & 0.524     & 0.494     \\
    & 3 & 0.526 & 0.515 & 0.515 & 0.518 & 0.517 & 0.515     & 0.515     & 0.519     \\
    & 4 & 0.529 & 0.544 & 0.545 & 0.529 & 0.543 & 0.544 & 0.544 & 0.530 \\
    & 5 & 0.544 & 0.549 & 0.549 & 0.528 & 0.548 & 0.549 & 0.549 & 0.529 \\
    & 6 & 0.510 & 0.512 & 0.512 & 0.549 & 0.512 & 0.512 & 0.512 & 0.548 \\ \midrule
    \multirow{6}{*}{10} 
    & 1 & 0.525 & 0.535 & 0.536 & 0.505 & 0.511 & 0.531 & 0.532 & 0.524 \\
    & 2 & 0.513 & 0.513 & 0.513     & 0.513     & 0.513 & 0.513 & 0.513 & 0.513     \\
   & 3 & 0.509 & 0.509 & 0.509 & 0.520 & 0.509 & 0.509     & 0.509     & 0.520     \\
    & 4 & 0.545 & 0.545 & 0.545 & 0.530 & 0.545 & 0.545 & 0.545 & 0.530 \\
    & 5 & 0.555 & 0.556 & 0.556 & 0.492 & 0.556 & 0.556 & 0.556 & 0.492 \\
    & 6 & 0.531 & 0.531 & 0.531 & 0.526 & 0.531 & 0.531 & 0.531 & 0.526 \\ \midrule
    \multirow{6}{*}{20} 
 & 1 & 0.525 & 0.535 & 0.536 & 0.505 & 0.511 & 0.531 & 0.532 & 0.524 \\
    & 2 & 0.513 & 0.513 & 0.513     & 0.513     & 0.513 & 0.513 & 0.513 & 0.513     \\
    & 3 & 0.509 & 0.509 & 0.509 & 0.520 & 0.509 & 0.509     & 0.509     & 0.520     \\
    & 4 & 0.545 & 0.545 & 0.545 & 0.530 & 0.545 & 0.545 & 0.545 & 0.530 \\
    & 5 & 0.555 & 0.555 & 0.555 & 0.492 & 0.555 & 0.555 & 0.555 & 0.492 \\
    & 6 & 0.531 & 0.531 & 0.531 & 0.526 & 0.531 & 0.531 & 0.531 & 0.526 \\ 
\end{longtable}
\end{center}

\begin{center}
\footnotesize
\setlength{\tabcolsep}{2.5pt} 
\begin{longtable}{cc cccc cccc}
\caption{Short Term Memory Analysis under Varying Perturbations using the values of Hurst exponent} \label{Table3} \\
    \toprule
    & & \multicolumn{4}{c}{\textbf{ $\mathscr{P}(.,1)$}} & \multicolumn{4}{c}{\textbf{$\mathscr{C}(0,.)$}} \\ 
    \cmidrule(lr){3-6} \cmidrule(lr){7-10}
    $\phi$ & $i$ & $\mathscr{P}(2.414,1)$ & $\mathscr{W}(1,1)$ & $\mathscr{W}(.5,1)$ & $\mathscr{F}(0,1,1)$ & $\mathscr{P}(2.414,1)$ & $\mathscr{W}(1,1)$ & $\mathscr{W}(.5,1)$ & $\mathscr{F}(0,1,1)$ \\
    \midrule
\endfirsthead

    \toprule
    & & \multicolumn{4}{c}{\textbf{ $\mathscr{P}(.,1)$}} & \multicolumn{4}{c}{\textbf{ $\mathscr{C}(0,j+1)$}} \\ 
    \cmidrule(lr){3-6} \cmidrule(lr){7-10}
    $\phi$ & $i$ & $\mathscr{P}(2.414,1)$ & $\mathscr{W}(1,1)$ & $\mathscr{W}(.5,1)$ & $\mathscr{F}(0,1,1)$ & $\mathscr{P}(2.414,1)$ & $\mathscr{W}(1,1)$ & $\mathscr{W}(.5,1)$ & $\mathscr{F}(0,1,1)$ \\
    \midrule
\endhead

    \multirow{6}{*}{0.1} 
    & 1 & 0.477 & 0.477 & 0.477     & 0.498     & 0.517 & 0.517     & 0.517     & 0.537     \\
    & 2 & 0.477 & 0.477 & 0.477     & 0.498      & 0.517 & 0.517     & 0.517     & 0.537     \\
    & 3 & 0.477 & 0.477 & 0.477     & 0.498     & 0.517 & 0.517     & 0.517     & 0.537    \\
    & 4 & 0.477 & 0.477 & 0.477     & 0.498      & 0.517 & 0.517     & 0.517     & 0.537     \\
    & 5 & 0.477 & 0.477 & 0.477     & 0.498     & 0.517 & 0.517     & 0.517     & 0.537    \\
    & 6 & 0.477 & 0.477 & 0.477     & 0.498  & 0.517 & 0.517     & 0.517     & 0.537 \\ \midrule
    
    \multirow{6}{*}{0.25} 
    & 1 & 0.465 & 0.465    & 0.465     & 0.505     & 0.499 & 0.499     & 0.499     & 0.551     \\
    & 2 & 0.465 & 0.465    & 0.465     & 0.505     & 0.499 & 0.499     & 0.499     & 0.551    \\
    & 3 &0.465 & 0.465    & 0.465     & 0.505     & 0.499 & 0.499     & 0.499     & 0.551     \\
    & 4 & 0.465 & 0.465    & 0.465     & 0.505    & 0.499 & 0.499     & 0.499     & 0.551    \\
    & 5 &0.465 & 0.465    & 0.465     & 0.505    & 0.499 & 0.499     & 0.499     & 0.551     \\
    & 6 & 0.465 & 0.465    & 0.465     & 0.505 & 0.499 & 0.499     & 0.499     & 0.551\\ \midrule
    
    \multirow{6}{*}{0.5} 
    & 1 & 0.490 & 0.489     & 0.489     & 0.514     & 0.506 & 0.504 & 0.504 & 0.524 \\
    & 2 &  0.490 & 0.489     & 0.489     & 0.514      &0.506 & 0.504 & 0.504 & 0.524 \\
    & 3 & 0.490 & 0.489     & 0.489     & 0.514      & 0.506 & 0.504 & 0.504 & 0.524    \\
    & 4 & 0.490 & 0.489     & 0.489     & 0.514     &0.506 & 0.504 & 0.504 & 0.524     \\
    & 5 &  0.490 & 0.489     & 0.489     & 0.514     & 0.505 & 0.504 & 0.504 & 0.524    \\
    & 6 &  0.490 & 0.489     & 0.489     & 0.514  & 0.506 & 0.504 & 0.504 & 0.524 \\
    \bottomrule
\end{longtable}
\end{center}

\begin{center}
\footnotesize
\setlength{\tabcolsep}{2.5pt}
\begin{longtable}{cc cccc cccc}
\caption{Long Term Memory Analysis under Varying Perturbations using the values of Hurst exponent} \label{Table4} \\
    \toprule
    & & \multicolumn{4}{c} {\textbf{$\mathscr{P}(.,1)$}}& \multicolumn{4}{c}{\textbf{$\mathscr{C}(0,.)$}} \\ 
    \cmidrule(lr){3-6} \cmidrule(lr){7-10}
    $\beta$ & $i$ & $\mathscr{P}(2.414,1)$ & $\mathscr{W}(1,1)$ & $\mathscr{W}(.5,1)$ & $\mathscr{F}(0,1,1)$ & $\mathscr{P}(2.414,1)$ & $\mathscr{W}(1,1)$ & $\mathscr{W}(.5,1)$ & $\mathscr{F}(0,1,1)$ \\
    \midrule
\endfirsthead

    \toprule
    & & \multicolumn{4}{c}{\textbf{Hurst Exponent: } $\mathscr{P}(.,1)$} & \multicolumn{4}{c}{\textbf{Hurst Exponent: } $\mathscr{C}(0,.)$} \\ 
    \cmidrule(lr){3-6} \cmidrule(lr){7-10}
   $\beta$ & $i$ & $\mathscr{P}(2.414,1)$ & $\mathscr{W}(1,1)$ & $\mathscr{W}(.5,1)$ & $\mathscr{F}(0,1,1)$ & $\mathscr{P}(2.414,1)$ & $\mathscr{W}(1,1)$ & $\mathscr{W}(.5,1)$ & $\mathscr{F}(0,1,1)$ \\
    \midrule
\endhead

    \multirow{6}{*}{2} 
    & 1 & 0.515 & 0.501 & 0.501 & 0.512 & 0.543 & 0.537 & 0.538 & 0.505 \\
    & 2 & 0.531 & 0.531 & 0.531 & 0.489 & 0.518 & 0.517 & 0.517 & 0.512 \\
    & 3 & 0.497 & 0.497 & 0.0.497 & 0.509 & 0.537 & 0.529 & 0.529 & 0.501 \\
    & 4 & 0.510 & 0.509 & 0.509 & 0.509 & 0.502 & 0.502 & 0.502 & 0.508 \\
    & 5 & 0.514 & 0.513 & 0.513 & 0.514 & 0.532 & 0.528 & 0.528 & 0.521 \\
    & 6 & 0.475 & 0.474 & 0.474 & 0.513 & 0.507 & 0.505 & 0.505 & 0.538 \\ \midrule

    \multirow{6}{*}{10} 
    & 1 & 0.520 & 0.495 & 0.495 & 0.513 & 0.540 & 0.534 & 0.535 & 0.498 \\
    & 2 & 0.532 & 0.532 & 0.532 & 0.487 & 0.517 & 0.517 & 0.517 & 0.507 \\
    & 3 & 0.498 & 0.498 & 0.498 & 0.528 & 0.532 & 0.532 & 0.532 & 0.503 \\
    & 4 & 0.495 & 0.495 & 0.495 & 0.511 & 0.495 & 0.495 & 0.495 & 0.513 \\
    & 5 & 0.499 & 0.499 & 0.499 & 0.531 & 0.530 & 0.530 & 0.530 & 0.515 \\
    & 6 & 0.478 & 0.478 & 0.478 & 0.499 & 0.531 & 0.531 & 0.531 & 0.520 \\ \midrule

    \multirow{6}{*}{20} 
    & 1 & 0.520 & 0.495 & 0.495 & 0.513 & 0.540 & 0.534 & 0.535 & 0.498 \\
    & 2 & 0.532 & 0.532 & 0.532 & 0.487 & 0.517 & 0.517 & 0.517 & 0.507 \\
    & 3 & 0.498 & 0.498 & 0.498 & 0.528 & 0.532 & 0.532 & 0.532 & 0.503 \\
    & 4 & 0.495 & 0.495 & 0.495 & 0.511 & 0.495 & 0.495 & 0.495 & 0.513 \\
    & 5 & 0.499 & 0.499 & 0.499 & 0.531 & 0.530 & 0.530 & 0.530 & 0.501 \\
    & 6 & 0.478 & 0.478 & 0.478 & 0.499 & 0.531 & 0.531 & 0.531 & 0.526 \\ 
    \bottomrule
\end{longtable}
\end{center}
It can be observed from Tables \ref{Table1} and \ref{Table2} that the values of the hurst exponent across a particular index under four different perturbations varies more when the underlying distribution is Cauchy rather than Pareto. The increased volatility of the hurst exponent under Cauchy noise highlights the sensitivity of long-range dependence estimators to distributions without moments, whereas the Pareto distribution - despite its heavy tail - retains a more stable extremal structure.
Whereas (see Tables \ref{Table3} and \ref{Table4}), independent Cauchy mixtures provide more stability under particular indices in case of short term memory retention and Pareto mixtures for long term memory retention. The overall observation is that the \enquote{memory} of a system is not just a function of time, but a function of the extremal structure of the noise and hence there is a fundamental trade-off between local volatility and global persistence driven by the tail-type of the distributions. Observations indicate a Tail-Induced Partition of Memory: while Cauchy-heavy mixtures exhibit superior stability in characterizing local fluctuations due to their symmetric extremal shocks, Pareto mixtures are more robust in preserving the structural integrity of long-term dependence. This suggests that in functional time series, the choice of the distributional prior should be dictated by the temporal scale of the intended inference

\begin{figure}[H]
\centering
\includegraphics[width=1\textwidth]{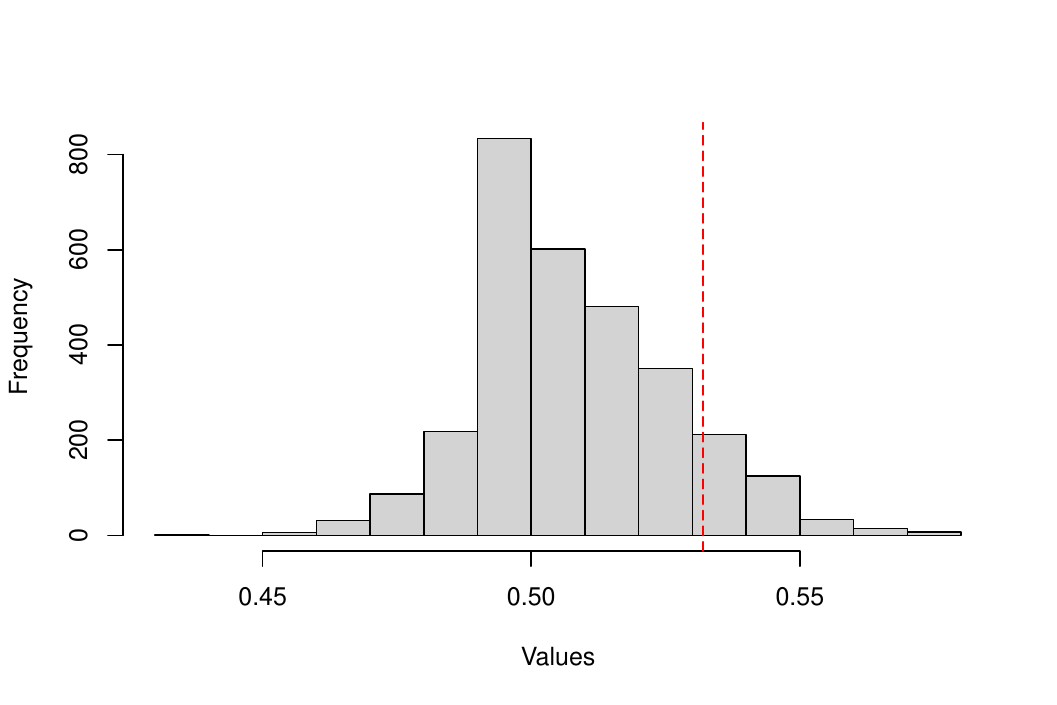}
\caption{Plot of the values of hurst exponent under Index (i=2) for $\beta$ =10 and $\mathscr{P}(3,1)$ innovation under $\mathscr{W}(0.5,1)$ perturbation}{\label{fig1}}
\end{figure}

It is an additional point of interest that the Hurst exponent for the non-iid setups are qualitatively similar to that of the iid settings, as predicted by our theoretical results in section \ref{TailCCTheory}. This is illustrated in Figure \ref{fig1}, which shows the plot of the hurst exponent for Index = 2 under iid $\mathscr{P}(3,1)$ with $\mathscr{W}(0.5,1)$ perturbation at $\beta=10$. The red dotted line is drawn at the value under non-iid setup of $\mathscr{P}$(.,1). Under all possible frameworks of both short memory and long memory, the estimates of the hurst exponent under the non iid case is contained within the band of the iid ones.

\section{Conclusion}
The quantitative analysis of financial time series is an important exercise for both academics as well as practitioners. Such analyses often reveal two distinct features that standard Gaussian frameworks fail to capture. One is the incidence of heavy-tailed marginal distributions and the second is the phenomenon of extreme co-movements. While extreme value theory characterizes marginal behavior, Copulas provide a convenient functional bridge to describe the dependence structure independently of the marginals. In this paper, we  proposed a different approach to looking at the joint extremes, through a novel dependence measure. The proposed idea incorporates both the non-identical and identical regularly varying distributions. The effect of persistence property have been thoroughly studied under these setups. 
The newly devised statistic seemed to be effective in determining the tail dependence. 

We have also carried out extensive simulations and the results obtained support our theoretical findings. An illustration of the same with real data has also been provided, using high-frequency cryptocurrency datasets on some common denominations. In the analysis of cryptocurrencies, we have observed that the three coins chosen indeed have tail dependence even at higher thresholds. At such thresholds, data becomes sparse so an obvious extension would be to look at its robust counterpart. An important question to ask is:  \enquote{When the data gets contaminated by outliers, heavy-tailed noise, or structural breaks, can still this dependence measure work or some loss quantification is required}. We will delve into these issues in future work.

\bibliography{sn-bibliography}
\section*{Appendix: Proofs}

    \begin{proof}[Proof of Theorem \ref{Thm}(a):]
    As $x \to \infty$, the tail probability of the linear combination of independent heavy-tailed random variables is dominated by the heaviest tail:
    \begin{align*}
        \mathbb{P}(\mathbf{l}'\mathbf{X} > x) &= \mathbb{P}(l_{1}X_{1} + \dots + l_{n}X_{n} > x) \\
        &\sim \mathbb{P}(l_{*}X_{*} > x)
    \end{align*}
    where $X_{*}$ is defined as the component $X_j$ such that its tail index is minimal:
    \[
    X_{*} = X_{j} \quad \text{where} \quad \alpha_{*} = \min_{j} \{\alpha_{j}\}
    \]
    \[
    \mathbb{P}(l_{*}X_{*} > x) \sim p(\frac{x}{l_{*}})^{-\alpha_{*}} L(\frac{x}{l_{*}}) 1(l_{*} > 0) + q(\frac{x}{|l_{*}|})^{-\alpha_{*}} L(\frac{x}{|l_{*}|}) 1(l_{*} < 0)
    \]
   
    Taking logarithm on both the sides yield the linear relationship in $\log(\frac{x}{l_{*}})$, consistent with power-law behavior.
    \end{proof}

     \begin{proof}[Proof of Theorem \ref{Thm}(b):]
\[
\mathbb{P}\left( n\max \{|l_{1}X_{1}| , \dots , |l_{n}X_{n}|\} > x \right) \geq \mathbb{P}(\mathbf{l}'\mathbf{X} > x)
\]
This implies,
\[
 \mathbb{P}(\mathbf{l}'\mathbf{X} \leq x) \geq \prod_{i=1}^{n}P(|l_iX_i| \leq \frac{x}{n}) 
\]

Hence,
\[
\log(\mathbb{P}(\mathbf{l}'\mathbf{X} > x) ) \leq log(1-\prod_{i=1}^{n}P(|l_iX_i|\leq \frac{x}{n}))=\sum_{i=1}^{n}-\alpha_{i} \log(\frac{x}{n|l^{*}|}) + \log(L(\frac{x}{n|l^*|}))
\]
where
\[
|l^{*}|=min|l_i| 
\]
and the overall product dominates  the various combinations of products.
     \end{proof}

 \begin{proof}[Proof of Theorem \ref{Thm}(c):]
 By Cauchy–Schwarz inequality 
 \[
 (\mathbf{l}'\mathbf{X})^{2}\leq(\mathbf{l}'\mathbf{l})(\mathbf{X}'\mathbf{X})
 \]
         \[
        \mathbb{P}(\mathbf{l}'\mathbf{X} > x) \leq  \mathbb{P}((\mathbf{l}'\mathbf{l})(\mathbf{X}'\mathbf{X}) >x^2)\leq \mathbb{P}(\sum_{i=1}^{n}X_i^2>\frac{x^2}{\sum_{i=1}^{n}l_i^2})\leq \sum_{i=1}^{n} \mathbb{P}(X_i^2 >\frac{x^2}{\sum_{i=1}^{n}l_i^2})
         \]
         Now, by tail balance condition we have,
         \[
         \sum_{i=1}^{n} \mathbb{P}(X_i^2 >\frac{x^2}{\sum_{i=1}^{n}l_i^2}) = \sum_{i=1}^{n} \mathbb{P}(|X_i|>\sqrt{\frac{x^2}{\sum_{i=1}^{n}l_i^2}})
         \] 
         \\
         \[
         log(\mathbb{P}(\mathbf{l}'\mathbf{X} > x)) \leq log(\sum_{i=1}^{n}\mathbb{P}(|X_i|>\sqrt{\frac{x^2}{\sum_{i=1}^{n}l_i^2})})=log\{a^{-\alpha_1}L_1(a)+a^{-\alpha_2}L_2(a)+..+a^{-\alpha_n}L_n(a)\}
         \] where, 
         \[
         a=\sqrt{\frac{x^2}{\sum_{i=1}^{n}l_i^2}}
         \]
 Define, \[ L(a)=\sup_{n}(L_n(a))\]  
         
         \[
         \leq log(L(a))+log\{a^{-\alpha_1}+..+a^{-\alpha_n}\} =log(L(a))+log(\frac{\sum_{i=1}^{n}a^{-\alpha_i}}{n}) + log(n)
         \]

         \[
     =log(L(a))+log(\frac{\sum_{i=1}^{n}e^{-\alpha_ilog(a)}}{n})+log(n)\]
     \[ \leq log(L(a))+log(\frac{\sum_{i=1}^{n}(1-\alpha_ilog(a)+\frac{1}{2}\alpha_i^2log(a)^2}{n})+log(n)
         \]
      Using,   \[
         log(1-a) \approx a \ \text{;} \ (log(a)^{2}) \approx log(a) \ \text{for large a}
        \] 
         \[
         \sim (\frac{\frac{1}{2}\sum_{i=1}^{n}\alpha_i^2-\sum_{i=1}^{n}\alpha_i}{n})log(a)+log(n)+log(L(a)) 
         \]
         \[
         =\sim (\frac{\frac{1}{2}\sum_{i=1}^{n}\alpha_i^2-\sum_{i=1}^{n}\alpha_i}{n})log(a)+log(n)+log(L(x)) 
         \] 
        as  L(.) is a regularly varying function at $\infty$.
     \end{proof}

     \begin{proof}[Proof of Theorem \ref{thm:IID1}:]
         
\[
\frac{P(b_0\epsilon_2^{(2)}+b_1\epsilon_1^{(2)}+b_2\epsilon_0^{(1)}+\ldots>y)}{P(a_0\epsilon_2^{(1)}+a_1\epsilon_1^{(1)}+a_2\epsilon_0^{(1)}+\ldots>x)} \leq \frac{P(|b_0\epsilon_2^{(2)}|+|b_1\epsilon_1^{(2)}|+|b_2\epsilon_0^{(1)}|+|.|+...>y)}{P(a_0\epsilon_2^{(1)}+a_1\epsilon_1^{(1)}+a_2\epsilon_0^{(1)}+\ldots>x)}
\]

\[
\leq
\frac{|b_2|^{\alpha_1}P(|\epsilon_0^{(1)}|>y)+\sum_{j=0, j \neq 2}^{\infty} |b_j|^{\alpha_2}P(|\epsilon_1^{(1)}|>y)}{\sum_{j=0}^{\infty}[p(b_j)_+^{\alpha_1}+q(b_j)_-^{\alpha_1}]P(|\epsilon_0^{(1)}|>x)}
\]

If x=y, then
\[
\tau_{x_n,y_n}(1)
=\frac{|b_2|^{\alpha_1}}{\sum_{j=0}^{\infty}[p(b_j)_+^{\alpha_1}+q(b_j)_-^{\alpha_1}]}+\frac{\sum_{j=0, j \neq 2}^{\infty} |b_j|^{\alpha_2}P(|\epsilon_1^{(2)}>x)}{\sum_{j=0}^{\infty}[p(b_j)_+^{\alpha_1}+q(b_j)_-^{\alpha_1}]P(|\epsilon_1^{(1)}>x)} 
\]
\[
=\frac{|b_2|^{\alpha_1}}{\sum_{j=0}^{\infty}[p(b_j)_+^{\alpha_1}+q(b_j)_-^{\alpha_1}]}+\frac{\sum_{j=0, j \neq 2}^{\infty} |b_j|^{\alpha_2}x^{-(\alpha_2-\alpha_1)}L_2(x)}{\sum_{j=0}^{\infty}[p(b_j)_+^{\alpha_1}+q(b_j)_-^{\alpha_1}]L_1(x)} 
\]

\[
\tau_{x_n,y_n}(1) \approx \frac{|b_2|^{\alpha_1}}{\sum_{j=0}^{\infty}[p(b_j)_+^{\alpha_1}+q(b_j)_-^{\alpha_1}]}+\frac{\sum_{j=0, j \neq 2}^{\infty} |b_j|^{\alpha_2}x^{-(\alpha_2-\alpha_1)}L_2(x)}{\sum_{j=0}^{\infty}[p(b_j)_+^{\alpha_1}+q(b_j)_-^{\alpha_1}]L_1(x)} 
\]
Similarly,
\[
 \tau_{x_n,y_n}(2)\approx \frac{|b_3|^{\alpha_1}}{\sum_{j=0}^{\infty}[p(b_j)_+^{\alpha_1}+q(b_j)_-^{\alpha_1}]}+\frac{\sum_{j=0, j \neq 3}^{\infty} |b_j|^{\alpha_2}x^{-(\alpha_2-\alpha_1)}L_2(x)}{\sum_{j=0}^{\infty}[p(b_j)_+^{\alpha_1}+q(b_j)_-^{\alpha_1}]L_1(x)}
\]
\\

Define
\[
\tau(2)-\tau(1)=\Delta
\]
Then,
\[
\Delta \propto (|b_3|^{\alpha_1}-|b_2|^{\alpha_1})+\frac{L_2(x)}{L_1(x)}x^{-(\alpha_2-\alpha_1)}|b_2|^{\alpha_2} 
\]
Now, under the following conditions \[
\tau(2)\leq \tau(1)  
\]
Exactly one of the following conditions must hold:

\begin{enumerate}
    \item $\alpha_2 > \alpha_1$
    \item The ratio of the slowly varying functions satisfies:
    \begin{equation}
        \frac{L_2(x)}{L_1(x)} = o\left(x^{-(\alpha_2 - \alpha_1)}\right) \quad \text{as } x \to \infty
    \end{equation}
\end{enumerate}

\noindent and additionally, the coefficients must satisfy:

\begin{enumerate}
    \setcounter{enumi}{2} 
    \item $\displaystyle \left| \frac{b_3}{b_2} \right| < 1$
\end{enumerate}

In general, for any $\tau(i+1)\leq \tau(i)$, condition (3) becomes \[
|\frac{b_{i+2}}{b_{i+1}}| < 1
\]

\end{proof}  

\begin{proof}[Proof of Lemma \ref{mem.lemma}]
\[
   \frac{\tau(2)}{\tau(1)} \approx 
\frac{\frac{|b_3|^{\alpha_1}}{D}+\frac{\sum_{j=0, j \neq 3}^{\infty}|b_j|^{\alpha_2}x^{-(\alpha_2-\alpha_1)}L_2(x)}{DL_1(x)}}  {\frac{|b_2|^{\alpha_1}}{D}+\frac{\sum_{j=0, j \neq 2}^{\infty} |b_j|^{\alpha_2}x^{-(\alpha_2-\alpha_1)}L_2(x)}{DL_1(x)}}
\approx \frac{|b_3|^{\alpha_1}+\delta(x)\sum_{j=0, j \neq 3}^{\infty}|b_j|^{\alpha_2}} {|b_2|^{\alpha_1}+\delta(x)\sum_{j=0, j \neq 2}^{\infty}|b_j|^{\alpha_2}}
\]

\[
\approx
|\frac{b_3}{b_2}|^{\alpha_1}[1+\delta(x)(\frac{\sum_{j=0, j \neq 3}^{\infty}|b_j|^{\alpha_1}}{|b_3|^{\alpha_1}}-   \frac{\sum_{j=0, j \neq 2}^{\infty}|b_j|^{\alpha_1}}{|b_2|^{\alpha_1}} ]
\]
 
Under the above noted conditions; 
\[
\frac{\tau(2)}{\tau(1)} \approx \frac{|b_3|}{|b_2|}^{\alpha_{1}}
\]
In general,
\[
\frac{\tau(i+1)}{\tau(i)} \approx \frac{|b_{i+2}|}{|b_{i+1}|}^{\alpha_{1}}
\]

Using the duality between AR(1) and MA$(\infty)$ representation:\\
Consider, 
\[
b_j=\phi^j \quad; \quad  |\phi|<1
\]
Hence,
\[
\frac{\tau(i+1)}{\tau(i)} \approx \frac{|b_{i+2}|}{|b_{i+1}|}^{\alpha_{1}} \approx |\phi|^{\alpha_1}
\]
This shows that for a standard linear process, the Tail Cross-Correlation Function decays exponentially with the lag $h$ at a rate of $|\phi|^{\alpha_1}$. Hence, short memory process.\\
On the contrary,

\[
b_j=j^{-\beta} 
\]
Hence,
\[
\frac{\tau(i+1)}{\tau(i)} \approx \frac{|b_{i+2}|}{|b_{i+1}|}^{\alpha_{1}} \approx \frac{|i+2|}{|i+1|}^{-\alpha_1 \beta}
\] 
\end{proof}
\begin{proof}[Proof of Theorem \ref{thm:NID1}]

\[
 \tau_{x_n,y_n}(1)\approx \frac{|b_2|^{\alpha_0}}{[p(a_1)_+^{\alpha_0}+q(a_1)_-^{\alpha_0}]}+\frac{\sum_{j=0, j \neq 2}^{\infty} |b_j|^{\alpha_{2-j}^{(2)}}P(|\epsilon_{2-j}^{(2)}|>x)}{[p(a_1)_+^{\alpha_0}+q(a_1)_-^{\alpha_0}]P(|\epsilon_0^{*}|>x)} 
\]
\[
 \tau_{x_n,y_n}(2)\approx \frac{|b_3|^{\alpha_0}}{[p(a_1)_+^{\alpha_0}+q(a_1)_-^{\alpha_0}]}+\frac{\sum_{j=0, j \neq 3}^{\infty}|b_j|^{\alpha_{3-j}^{(2)}}P(|\epsilon_{3-j}^{(2)}|>x)}{[p(a_1)_+^{\alpha_0}+q(a_1)_-^{\alpha_0}]P(|\epsilon_0^{*}|>x)} 
\]

Define, $\delta = \tau(2) -\tau(1)$
Then,
\[
\delta \approx \frac{|b_3|^{\alpha_0}-|b_2|^{\alpha_0}}{[p(a_1)_+^{\alpha_0}+q(a_1)_-^{\alpha_0}]}+\frac{\sum_{j=0, j \neq 3}^{\infty} |b_j|^{\alpha_{3-j}^{(2)}}P(|\epsilon_{3-j}^{(2)}|>x)-\sum_{j=0, j \neq 2}^{\infty} |b_j|^{\alpha_{2-j}^{(2)}}P(|\epsilon_{2-j}^{(2)}|>x)}{[p(a_1)_+^{\alpha_0}+q(a_1)_-^{\alpha_0}]P(|\epsilon_0^{*}|>x)} 
\]

In general,
\[
\delta(i+1,i)=\tau(i+1)-\tau(i) \approx \frac{|b_{i+2}|^{\alpha_0}-|b_{i+1}|^{\alpha_0}}{[p(a_1)_+^{\alpha_0}+q(a_1)_-^{\alpha_0}]}+\frac{\sum_{j=0, j \neq i+1}^{\infty} |b_j|^{\alpha_{i+1-j}^{(2)}}P(|\epsilon_{i+1-j}^{(2)}|>x)-\sum_{j=0, j \neq i}^{\infty}|b_j|^{\alpha_{i-j}^{(2)}}P(|\epsilon_{i-j}^{(2)}|>x)}{[p(a_1)_+^{\alpha_0}+q(a_1)_-^{\alpha_0}]P(|\epsilon_0^{*}|>x)} 
\]

Exactly one must hold
 
\[
 \text{(1)} \quad  \alpha_i^{(2)} < \quad \alpha_0 \quad \forall \quad  i \]
 \[
 \text{(2)}  \quad \frac{L_i(x)}{L(x)} =o(x^{-(\alpha_i^{(2)}-\alpha_0}) \quad \forall \quad i
\] 
\text{(3)}
Combination of  conditions (1) and  (2). 
Under the following assumptions;
\[
\delta <1 
\] if
\[
\frac{|b_3|^{\alpha_0}-|b_2|^{\alpha_0}}{[p(a_1)_+^{\alpha_0}+q(a_1)_-^{\alpha_0}]} <1
\]

In general, 
\[
\tau(i+1)-\tau(i) < 1
\]
if
\[
\frac{|b_{i+2}|^{\alpha_0}-|b_{i+1}|^{\alpha_0}}{[p(a_1)_+^{\alpha_0}+q(a_1)_-^{\alpha_0}]}<1
\] along with the assumptions.
\end{proof}

\[
\frac{\tau(i+1)}{\tau(i)} = \frac{|b_{i+2}|^{\alpha_0} P\left(|\epsilon_0^{*}| > x\right) + \sum_{\substack{j=0 \\ j \neq i+2}}^{\infty} |b_j|^{\alpha_{i+1-j}^{(2)}} P\left(|\epsilon_{i+1-j}^{(2)}| > x\right)}{|b_{i+1}|^{\alpha_0} P\left(|\epsilon_0^{*}| > x\right) +  \sum_{\substack{j=0 \\ j \neq i+1}}^{\infty} |b_j|^{\alpha_{i-j}^{(2)}} P\left(|\epsilon_{i-j}^{(2)}| > x\right)}
\]
Under conditions (1) and (2)

\begin{equation}
\frac{\tau(i+1)}{\tau(i)} \approx \left| \frac{b_{i+2}}{b_{i+1}} \right|^{\alpha_0}
\end{equation}

Hence, equivalence with Lemma \ref{mem.lemma}.
\begin{proof}[Proof of Theorem \ref{thm:IID2}]
    \[
\frac{P(b_0\epsilon_2^{(2)}+b_1\epsilon_1^{(2)}+b_2\epsilon_0^{(1)}+\ldots>y)}{P(a_0\epsilon_2^{(1)}+a_1\epsilon_1^{(1)}+a_2\epsilon_0^{(1)}+\ldots>x)} \leq \frac{P(|b_0\epsilon_2^{(2)}|+|b_1\epsilon_1^{(2)}|+|b_2\epsilon_0^{(1)}|+|.|+...>y)}{P(a_0\epsilon_2^{(1)}+a_1\epsilon_1^{(1)}+a_2\epsilon_0^{(1)}+\ldots>x)}
\]

\[
\leq
\frac{|b_2|^{\alpha_1}P(|\epsilon_0^{(1)}|>y)+\sum_{j=0, j \neq 2}^{\infty} |b_j|^{\alpha_2}P(|\epsilon_1^{(1)}|>y)}{\sum_{j=0}^{\infty}[p(b_j)_+^{\alpha_1}+q(b_j)_-^{\alpha_1}]P(|\epsilon_0^{(1)}|>x)}
\]
\[
\tau_{x_n,y_n}(1) \approx \frac{|b_2|^{\alpha_1}P(|\epsilon_0^{(1)}|>y)}{\sum_{j=0}^{\infty}[p(b_j)_+^{\alpha_1}+q(b_j)_-^{\alpha_1}]P(|\epsilon_0^{(1)}|>x)}+\frac{\sum_{j=0, j \neq 2}^{\infty} |b_j|^{\alpha_2}P(|\epsilon_1^{(1)}|>y)}{\sum_{j=0}^{\infty}[p(b_j)_+^{\alpha_1}+q(b_j)_-^{\alpha_1}]P(|\epsilon_0^{(1)}|>x)} 
\]
Similarly,
\[
 \tau_{x_n,y_n}(2)\approx \frac{|b_3|^{\alpha_1}P(|\epsilon_0^{(1)}|>y)}{\sum_{j=0}^{\infty}[p(b_j)_+^{\alpha_1}+q(b_j)_-^{\alpha_1}]P(|\epsilon_0^{(1)}|>x)}+\frac{\sum_{j=0, j \neq 3}^{\infty} |b_j|^{\alpha_2}P(|\epsilon_1^{(1)}|>y)}{\sum_{j=0}^{\infty}[p(b_j)_+^{\alpha_1}+q(b_j)_-^{\alpha_1}]P(|\epsilon_0^{(1)}|>x)}
\]
\end{proof}
Therefore, 
\[
\begin{gathered}
\delta(i+1,i)=\tau(i+1)-\tau(i) \\
\approx \frac{(|b_{i+2}|^{\alpha_0}-|b_{i+1}|^{\alpha_0})P(|\epsilon_0^{(1)}|>y)}{\sum_{j=0}^{\infty}[p(b_j)_+^{\alpha_1}+q(b_j)_-^{\alpha_1}]P(|\epsilon_0^{(1)}|>x)}+\frac{\sum_{j=0, j \neq i+1}^{\infty} (|b_j|^{\alpha_{i+1-j}^{(2)}}-\sum_{j=0, j \neq i}^{\infty}|b_j|^{\alpha_{i-j}^{(2)}})P(|\epsilon_1^{(1)}|>y)}{\sum_{j=0}^{\infty}[p(b_j)_+^{\alpha_1}+q(b_j)_-^{\alpha_1}]P(|\epsilon_0^{(1)}|>x)} 
\end{gathered}
\]
\[
= \frac{|b_{i+2}|^{\alpha_1} - |b_{i+1}|^{\alpha_1}}{\sum_{j=0}^{\infty} \left[ p(b_j)_{+}^{\alpha_1} + q(b_j)_{-}^{\alpha_1} \right]} \left( \frac{x}{y} \right)^{\alpha_1} \frac{L(y)}{L(x)} + \frac{|b_{i+1}|^{\alpha_2} - |b_{i+2}|^{\alpha_2}}{\sum_{j=0}^{\infty}  p(b_j)_{+}^{\alpha_1} + q(b_j)_{-}^{\alpha_1}} \frac{y^{-\alpha_2}}{x^{-\alpha_1}} \frac{L_1(y)}{L(x)}
\]
Hence, the conditions.\\
Ratio follows the same analogy as Theorem \ref{thm:IID1}.

\begin{proof}[Proof of Theorem \ref{thm:NID2}]
\[
\begin{gathered}
\delta(i+1,i)=\tau(i+1)-\tau(i) \\ \approx \frac{|b_{i+2}|^{\alpha_0}-|b_{i+1}|^{\alpha_0}P(|\epsilon_0^{*}|>y)}{[p(a_1)_+^{\alpha_0}+q(a_1)_-^{\alpha_0}]P(|\epsilon_0^{*}|>x)}+\frac{\sum_{j=0, j \neq i+2}^{\infty} |b_j|^{\alpha_{i+1-j}^{(2)}}P(|\epsilon_{i+1-j}^{(2)}|>y)-\sum_{j=0, j \neq i+1}^{\infty}|b_j|^{\alpha_{i-j}^{(2)}}P(|\epsilon_{i-j}^{(2)}|>y)}{[p(a_1)_+^{\alpha_0}+q(a_1)_-^{\alpha_0}]P(|\epsilon_0^{*}|>x)} 
\end{gathered}
\]
Hence, the conditions.\\
Ratio follows the same analogy as Theorem \ref{thm:NID1}.

\end{proof}
\end{document}